\documentclass{svjour3}                     
\smartqed  
\usepackage{graphicx}
\usepackage[latin1]{inputenc}
\usepackage{amsfonts,dsfont}
\usepackage{amssymb,bm}
\usepackage{amsmath,enumerate,rotate}
\usepackage{amssymb,latexsym}
\usepackage{epsfig}
\usepackage{graphicx}
\usepackage[english]{babel}
\usepackage{color}

\usepackage{verbatim}

\newcommand{\E}{\mathds{E}}

\newcommand{\ddr}{\mathrm{d}}
\newcommand{\edr}{\mathrm{e}}

\newtheorem{thm}{\noindent Theorem}

\newtheorem{prop}{\noindent Proposition}

\begin{document}

\title{Limiting behavior of the search cost distribution for the move-to-front rule in the stable case
}

\titlerunning{Limiting behavior of the search cost distribution in the stable case}        

\author{F. Leisen \and A. Lijoi \and C. Paroissin
}


\institute{Fabrizio Leisen \at
              Universidad de Navarra, Faculdad de Ciencias Economicas y Empresariales, Campus Universitario, edificio de biblioteca (entrada este), 31080, Pamplona, Spain.               \\
\email{fabrizio.leisen@gmail.com} 
           \and
      Antonio Lijoi \at
              Universit\`a degli Studi di Pavia, Dipartimento di Economia Politica, Via San
Felice 5, 27100 Pavia, Italy. \\
              \email{lijoi@unipv.it} \at
Collegio Carlo Alberto, Via Real Collegio 30, 10024 Moncalieri, Italy. \\
          \and
Christian Paroissin\at
Universit\'e de Pau et des Pays de l'Adour, Laboratoire de Math\'ematiques et de leurs Applications - UMR CNRS 5142, Avenue de l'Universit\'e,
64013 Pau cedex, France.\\
\email{christian.paroissin@univ-pau.fr}
}

\date{Received: date / Accepted: date}

\maketitle

\begin{abstract}
Move-to-front rule is a heuristic updating a list of $n$ items according to requests. Items are required with unknown probabilities (or popularities). The induced Markov chain is known to be ergodic \cite{donnelly}. One main problem is the study of the distribution of the search cost defined as the position of the required item. Here we first establish the link between two recent papers \cite{bp,lp} that both extend results proved by Kingman \cite{kingman} on the expected stationary search cost. Combining results contained in these papers, we obtain the limiting behavior for any moments of the stationary seach cost as $n$ tends to infinity.

\keywords{normalized random measure \and random discrete distribution\and stable subordinator \and problem of heaps}
 \subclass{MSC 60G57 \and MSC 60G51 \and MSC 60G52}
\end{abstract}

\section{Introduction}
\label{sec:intro}

The heaps problem was first considered, in independent works, by Tsetlin \cite{tsetlin} and McCabe \cite{McCabe}. Its basic description can be given as follows. Consider a collection of $n$ items stored into a list or heap and each of them is identified by a label. Hence, the objects can be described by the set $I= \{ 1, \ldots, n\}$. The probability that the $i$--th item is requested by a user is denoted by $p_i$, for $i=1,\ldots,n$. Hence $p_i \geqslant 0$, for any $i$, and $\sum_{i=1}^{n}p_i=1$. At each unit of time, an item is requested and it is searched for through the heap, starting at the top. Once it is found, it is moved to the top of the heap. The search cost is the position of the requested item in the heap or, equivalently, the number of items to be removed from the heap in order to find the requested one. In this setting, it might be of interest to determine the distribution of the search cost when the underlying Markov chain is at equilibrium.
\\[1ex]
Kingman \cite{kingman} first studied the case of random request probabilities, or random popularities. His paper develops two important cases where request probabilities are defined in terms of: (a) the normalized increments of a $\gamma$-stable subordinator; (b)  the Dirichlet distribution on the simplex. The results contained therein provide an exact analytic evaluation of the expected search cost either for any finite $n$ or in the limit, as the number of items $n$ tends to infinity. In particular, in the case of normalized $\gamma$-stable request probabilities, it is found that the limiting expected search cost is finite if and only if $\gamma<1/2$.  
\\[1ex]
These results have been recently extended in two independent papers. In \cite{lp} Lijoi and Pr\"unster  studied the case of request probabilities derived from a normalized random measures with independent increments, which generalizes the result obtained by Kingman \cite{kingman}. In \cite{bp} Barrera and Paroissin  studied the case of request probabilities based on exchangeable random partitions.
\\[1ex]
It is to be emphasized that all previous contributions on the subject is confined to the determination of the first moment of the stationary search cost. Here we wish to extend earlier work and determine the expression of the limiting moments of any order in the $\gamma$--stable case. In particular, it will be shown that the $k$--th moment exists if and only if $\gamma<1/(k+1)$ which reduces to the condition provided by \cite{kingman} when $k=1$. See also \cite{lp}. The outline of the paper is as follows. In Section~2 we provide a concise introduction to some basic tools and notions that will be relevant for achieving the main result in Section~3.

\section{The $\gamma$-stable model}
\label{sec:gamma_model}

Before stating and proving our result, it might be worth recalling the main ingredients that define the model we are going to use. As mentioned in the previous section, the request probabilities $p_i$, for $i=1,\ldots,n$, are going to be random. Indeed, if $(w_i)_{i\ge 1}$ is a sequence of positive independent random variables and $W_n=\sum_{i=1}^n w_i$, one can define
\[
p_i=\frac{w_i}{W_n} \qquad i=1,\ldots,n
\]  
Hence, $(p_1,\ldots,p_n)$ is an exchangeable random partition of the unit interval. A possible choice is $w_i:=\xi_{t_{i}}-\xi_{t_{i-1}}$ where $0=t_0< t_1<\,\cdots\,<t_n=1$ and $\xi=\{\xi_t:\: t\in[0,1]\}$ is a subordinator that is a process with almost surely increasing paths and with independents increments. In this case, one can express the the Laplace transform of $w_i$ in terms of the L\'evy intensity $\nu$ of $\xi$. In other words
\begin{equation}
  \label{eq:lapl_w}
  \phi_i(s):=\E\left[\edr^{-s w_i}\right]=\exp\left\{-(t_i-t_{i-1})\int_0^\infty\left[1-\edr^{-sy}\right]\,
\nu(\ddr y)\right\}
\end{equation}
with $\nu$ such that $\int_0^\infty\min\{1,y\}\,\nu(\ddr y)<\infty$. According to the terminology set forth in \cite{rlp}, $(p_1,\ldots,p_n)$ defines a normalized random measure with independent increments (NRMI). \\[1ex] 
Lijoi and Pr\"unster in \cite{lp} considered this general construction to determine an expression of the expected value of the search cost $S_n$. In the special case where 
\begin{equation}
  \label{eq:gamma_st_levy}
  \nu(\ddr y)=\frac{\gamma}{\Gamma(1-\gamma)}\: y^{-1-\gamma}\:\ddr y \qquad \gamma\in(0,1)
\end{equation}
they recovered an expression of the limiting expected search cost, as $n$ tends to infinity, thus recovering a result proved by \cite{kingman}. Note that if $\nu$ is as in \eqref{eq:gamma_st_levy}, then $\phi_i(s)=\exp\{-(t_i-t_{i-1})s^\gamma\}$ for any $s\ge 0$.
\\[1ex]
Barrera and Paroissin \cite{bp} have been able to determine an integral representation for the Laplace transform $\phi_{S_n}$ of the search cost $S_n$ in terms of the Laplace transforms $\phi_i$ of the single random weights $w_i$. In doing so they rely on results proved by Fill and Holst \cite{fh}. The expression they obtain is, then, used to derive a formula for the first two moments. From these formulas, they get an asymptotic equivalent for the Laplace transform of $S_n$ and the limit of the two first moments. Only this last point needs the assumption that the expectation of $S_n$ is finite. Two examples are studied: the case of deterministic weight and the case of gamma weight, which corresponds to the Dirichlet partition. Notice that, for this case, some limiting results were proved with an alternative way in \cite{bhp}. The limiting distribution has been also derived in the general iid case provided that the expectation $\mu_i$ of $w_i$ is finite \cite{bhp2}.
\\[1ex]
In the following section we will undertake the approach developed in \cite{bp} and determine the $k$--th moment of $S_n$ by working directly on $\phi_{S_n}$.

\section{Moments of the stationary search cost}
\label{sec:mom2}

The main tool we are relying on for the evaluation of $\E[S_n^k]$ is the Laplace transform of $S_n$ as displayed in theorem~2.2 of \cite{bp} and recalled here below.

\begin{thm}
For a sequence $(w_i)_{i\ge 1}$ of independent random variables
\begin{equation}\label{eqn:lt}
\phi_{S_n}(s) = \sum_{i=1}^n \int_0^\infty \left( \int_t^\infty \phi_i''(r) \prod_{j \neq i}h_{t,s,j}(r) \,\ddr r \right) \,\ddr t \;,
\end{equation}
for all $s \geqslant 0$, where for all $j \in \{1, \ldots, n\}$,
$$
h_{t,s,j}(r) = \phi_j(r)+e^{-s} (\phi_j(r-t)-\phi_j(r)) \;, \quad t \geqslant 0, r \geqslant 0 \;.
$$
\end{thm}

Using \eqref{eqn:lt}, we are able to compute moments of any order of the search cost $S_n$. Before doing so we need to introduce the quantity
\begin{multline}
  \label{eq:derivatives}
  M_{k,n}(s):=\edr^{-ks}\sum_{i\ne i_1\ne\,\cdots\,\ne i_k}\int_0^\infty\ddr t\,\int_t^\infty \ddr r \:\phi''_i(r)\,\prod_{l=1}^k (\phi_{i_l}(r-t)-\phi_{i_l}(r))\\
\times\prod_{j\not\in\{ i,i_1,\dots i_k\}}[\phi_j(r)+e^{-s}(\phi_j(r-t)-\phi_j(r))]
\end{multline} 
whose values, at $s=0$, will determine the moments of $S_n$.

\begin{prop}\label{crucial}
If the $(p_1,\ldots,p_n)$ are determined by normalizing the increments of a $\gamma$-stable subordinator with $t_i-t_{i-1}=1/n$ in \eqref{eq:lapl_w} for each $i \in \{1, \ldots, n\}$, then
$$
\lim_{n \rightarrow \infty} M_{k,n}(0) = \left\{ \begin{array}{cc}
\frac{(k!)^2}{(\frac{1}{\gamma}-k-1)_{k}} & \:\:\:{\mbox{\rm if }} \gamma < \frac{1}{k+1} \\[7pt]
\infty & \:\:\:{\mbox{\rm otherwise}}
\end{array}\right.
$$
where $(a)_k=\Gamma(a+k)/\Gamma(a)$ is the $k$--th ascending factorial of $a$.
\end{prop}

\begin{proof}
Note first that $\phi_i(s)=\exp\{-s^\gamma/n\}$ for any $s\ge 0$. Moreover
\begin{align*}
M_{k,n}(0)
&= \sum_{i\ne i_1\ne \cdots \ne i_k} \int_0^{\infty}\int_t^{\infty}\phi''_i(r)\:\prod_{l=1}^k(\phi_{i_l}(r-t)-\phi_{i_l}(r))\:
\prod_{j\not\in \{i,i_1,\dots i_k\}}\phi_j(r-t)\:\ddr r\,\ddr t\\
&=\sum_{i\ne i_1\ne \cdots \ne i_k} \int_0^{\infty}\int_0^{\infty}\phi''_i(r+t)\:\prod_{l=1}^k(\phi_{i_l}(r)-\phi_{i_l}(r+t))\:
\prod_{j\not\in \{i,i_1,\dots i_k\}}\phi_j(r)\:\ddr r\,\ddr t\\
&=\sum_{i\ne i_1\ne \cdots \ne i_k}
\int_0^{\infty}\prod_{j\not\in\{i,i_1,\dots i_k\}}\phi_j(r)\int_0^{\infty}\phi'_i(r+t)\,
\sum_{l=1}^{k}\phi'_{i_l}(r+t)\\
&\quad\qquad \times\:\prod_{\substack{m=1\\m\neq l}}^k(\phi_{i_m}(r)-\phi_{i_m}(r+t))\ddr t\ddr r
\end{align*}
Taking into account the form of $\phi_i$ in the $\gamma$--stable case, one has
\begin{align*}
\prod_{\substack{m=1\\m\neq l}}^k 
 (\phi_{i_m}(r)-\phi_{i_m}(r+t))
&=\sum_{\bm{a}_l\in\{0,1\}^{k-1}}\:
\prod_{\substack{m=1\\m\neq l}}^k(-1)^{a_m}\phi_{i_m}^{a_m}(r+t)\phi_{i_m}^{1-a_m}(r)\\
&=\sum_{\bm{a}_l\in\{0,1\}^{k-1}}(-1)^{|\bm{a}_l|}\edr^{-\frac{(r+t)^{\gamma}}{n} |\bm{a}_l|}\,
\edr^{-\frac{r^{\gamma}}{n} (k-1-|\bm{a}_l|)}
\end{align*}
where $\bm{a}_l=(a_1,\ldots,a_{l-1},a_{l+1},\ldots,a_k)$ and $|\bm{a}_l|=\sum_{m\ne l}a_m$. Summing up, in the $\gamma$--stable case one has
\begin{align*}
M_{k,n}(0)
&=\frac{\gamma^2}{n^2}\,
\sum_{i\ne i_1\ne \cdots \ne i_k}\sum_{l=1}^{k}\sum_{\bm{a}_l\in\{0,1\}^{k-1}}
(-1)^{|\bm{a}_l|}\int_0^{\infty}\edr^{-r^{\gamma}(1-\frac{1}{n}-\frac{k}{n})}
\\
& \qquad \times\: \int_0^{\infty} (r+t)^{2\gamma-2}\,\edr^{-\frac{(r+t)^{\gamma}}{n}(2+|\bm{a}_l|)}\,\edr^{-\frac{r^{\gamma}}{n}(k-1-|\bm{a}_l|)}
\:\ddr t\,\ddr r\\
&=\frac{\gamma^2}{n^2}\,
\sum_{i\ne i_1\ne \cdots \ne i_k}\sum_{l=1}^{k}\sum_{\bm{a}_l\in\{0,1\}^{k-1}}
(-1)^{|\bm{a}_l|}\int_0^{\infty}\edr^{-r^{\gamma}(1-\frac{2}{n}-\frac{|\bm{a}_l|}{n})}\\
&\qquad \times\: \int_0^{\infty} (r+t)^{2\gamma-2}e^{-\frac{(r+t)^{\gamma}}{n}(|\bm{a}_l|+2)}\:\ddr t\,
\ddr r
\end{align*}
The change of variable $(x,y)=((r+t)^\gamma,r^\gamma)$ yields
\begin{align*}
M_{k,n}(0)
&= \frac{1}{n^2}
\,\sum_{i\ne i_1\ne \cdots \ne i_k}\sum_{l=1}^{k}\sum_{\bm{a}_l\in\{0,1\}^{k-1}}
(-1)^{|\bm{a}_l|}\int_0^{\infty}y^{\frac{1}{\gamma}-1}\edr^{-y(1-\frac{2}{n}-\frac{1}{n}|\bm{a}_l|)}\\
&\qquad \times\:\int_y^{\infty} x^{1-\frac{1}{\gamma}}e^{-\frac{x}{n}(2+|\bm{a}_l|)}\:\ddr x\,\ddr y
\end{align*}
Using formulae (3.381.6) and (7.621.3) in \cite{gr}, one finds out that
\begin{align*}
M_{k,n}(0)
&=\frac{1}{n^2}
\,\sum_{i\ne i_1\ne \cdots \ne i_k}\sum_{l=1}^{k}\sum_{\bm{a}_l\in\{0,1\}^{k-1}}
(-1)^{|\bm{a}_l|}\int_0^{\infty}
y^{\frac{1}{2\gamma}-\frac{1}{2}}\,
(n^{-1}(|\bm{a}_l|+2))^{-\frac{3}{2}+\frac{1}{2\gamma}}\\
&\qquad \times\: \edr^{-y(1-\frac{1}{n}-\frac{1}{2n}|\bm{a}_l|)}\:
W_{\frac{1}{2}-\frac{1}{2\gamma},1-\frac{1}{2\gamma}}
\left(y\frac{|\bm{a}_l|+2}{n}\right)\:\ddr y\\
&=
\frac{\gamma}{n^2}\sum_{i\ne i_1\ne \cdots \ne i_k}\sum_{l=1}^{k}\sum_{\bm{a}_l\in\{0,1\}^{k-1}}
(-1)^{|\bm{a}_l|} {}_2F_1\left(2,1;\frac{1}{\gamma}+1;1-\frac{2+|\bm{a}_l|}{n}\right)\\
&= \frac{\gamma}{n^2}\sum_{i\ne i_1\ne \cdots \ne i_k}\sum_{l=1}^{k}\,
\sum_{r=0}^{k-1}(-1)^{r}\binom{k-1}{r}\,
{}_2F_1\left(2,1;\frac{1}{\gamma}+1;1-\frac{2+r}{n}\right)\\
&=\frac{\gamma k(n-1)(n-2)\cdots(n-k)}{n}\sum_{r=0}^{k-1}(-1)^{r}\binom{k-1}{r}{}_2F_1
\left(2,1;\frac{1}{\gamma}+1;1-\frac{2+r}{n}\right)
\end{align*}
Since the Gauss hypergeometric function $ _2F_1$ can be rewritten as 
\[
{}_2F_1\left(2,1;\frac{1}{\gamma}+1;1-\frac{2+r}{n}\right)=\sum_{l=0}^{\infty}\frac{(2)_l (1)_l}{l!(1+\frac{1}{\gamma})_l}\sum_{j=0}^{l}(-1)^j \binom{l}{j} \left(\frac{2+r}{n}\right)^j
\]
the expression of $M_{k,n}(0)$ can be further simplified as follows
\begin{align*}
M_{k,n}(0) 
& = \frac{\gamma k(n-1)(n-2)\cdots(n-k)}{n}\sum_{r=0}^{k-1}(-1)^{r}\binom{k-1}{r}\\
&\qquad\qquad\qquad\times\: \sum_{j=0}^{\infty}\sum_{l=j}^{\infty}\frac{(2)_l (1)_l}{l!(1+\frac{1}{\gamma})_l}(-1)^j \binom{l}{j} \left(\frac{2+r}{n}\right)^j\\
&= \frac{\gamma k(n-1)(n-2)\cdots(n-k)}{n}\sum_{r=0}^{k-1}(-1)^{r}\binom{k-1}{r} \sum_{j=0}^{\infty}(-1)^j a_j \left(\frac{2+r}{n}\right)^j\\
\end{align*}
where
\begin{equation*}
a_j=\sum_{l=j}^{\infty}\frac{(2)_l (1)_l}{l!(1+\frac{1}{\gamma})_l} \binom{l}{j}
\end{equation*}
A simple change of variable $m=l-j$ leads to write $a_j$ as
\begin{eqnarray*}
a_j &=& \sum_{m=0}^{\infty}\frac{(2)_{m+j} (1)_{m+j}}{(m+j)!(1+\frac{1}{\gamma})_{m+j}}\binom{m+j}{j}\\
&=& \sum_{m=0}^{\infty}\frac{(m+j+1)! (m+j)!}{j!(1+\frac{1}{\gamma})_{m+j}}\frac{1}{m!}\\
&=&\frac{(j+1)!}{(1+\frac{1}{\gamma})_j}\sum_{m=0}^{\infty}\frac{(j+2)_m (j+1)_m}{(j+1+\frac{1}{\gamma})_m}\frac{(1)^m}{m!}\\
&=&\frac{(j+1)!}{(1+\frac{1}{\gamma})_j} {}_2F_1(j+2,j+1,j+1+\frac{1}{\gamma},1)\\
\end{eqnarray*}
and, consequently,
\begin{multline*}
M_{k,n}(0)=\frac{\gamma k(n-1)(n-2)\cdots(n-k)}{n}\sum_{r=0}^{k-1}(-1)^{r}\binom{k-1}{r} \sum_{j=0}^{\infty}(-1)^j \frac{(j+1)!}{(1+\frac{1}{\gamma})_j}\\
\times\: {}_2F_1(j+2,j+1,j+1+\frac{1}{\gamma},1) \left(\frac{2+r}{n}\right)^j
\end{multline*}
If one resorts to identity (0.154.6) in \cite{gr}, it follows that
\begin{multline*}
M_{k,n}(0)=\frac{\gamma k(n-1)(n-2)\cdots(n-k)}{n}\sum_{j=k-1}^{\infty}(-1)^j \frac{(j+1)!}{(1+\frac{1}{\gamma})_j}\\ 
\times\: {}_2F_1(j+2,j+1,j+1+\frac{1}{\gamma},1)\: \sum_{r=0}^{k-1}(-1)^{r}\binom{k-1}{r} \left(\frac{2+r}{n}\right)^j
\end{multline*}
Finally, using formula (0.154.5) in \cite{gr} one has
\begin{multline*}
M_{k,n}(0)=\frac{\gamma k(n-1)(n-2)\cdots(n-k)}{n}\left[\frac{1}{n^{k-1}} \frac{k!(k-1)!}{(1+\frac{1}{\gamma})_{k-1}} {}_2F_1(k+1,k,k+\frac{1}{\gamma},1)\right.\\
\left. +o(\frac{1}{n^{k-1}})\right]
\end{multline*}
as $n\to\infty$. 
If $\gamma<\frac{1}{k+1}$ then ${}_2F_1(k+1,k,k+\frac{1}{\gamma},1)=\frac{\Gamma(k+\frac{1}{\gamma})\Gamma(\frac{1}{\gamma}-k-1)}{\Gamma(\frac{1}{\gamma}-1)\Gamma(\frac{1}{\gamma})}$. Otherwise the Gauss hypergeometric function diverges (see paragraph 9.102 in \cite{gr}). After some little algebra the results is proved. \hfill $\square$
\end{proof}

The study of the limiting behavior of $M_{k,n}(0)$ is crucial for understanding the limiting behavior of the moments. Indeed,
$$E(S_n^k)=(-1)^k\left.\phi^{(k)}_{S_n}(s)\right |_{s=0}$$
In particular, we have:
\begin{eqnarray*}
E(S_n)&=& M_{1,n}(0)\\
E(S_n^2)&=& M_{1,n}(0)+M_{2,n}(0)\\
E(S_n^3)&=& M_{1,n}(0)+3M_{2,n}(0)+M_{3,n}(0)\\
E(S_n^4)&=& M_{1,n}(0)+7M_{2,n}(0)+6M_{3,n}(0)+M_{4,n}(0)\\
E(S_n^5)&=& M_{1,n}(0)+15M_{2,n}(0)+25M_{3,n}(0)+10M_{4,n}(0)+M_{5,n}(0)\\
\cdots &&\\\
\end{eqnarray*}
In general
\begin{equation}\label{allmoments}
E(S_n^k)=a_1^{(k)}M_{1,n}(0)+\cdots + a_k^{(k)}M_{k,n}(0)
\end{equation}
where 
\begin{equation}\label{ricor}
\begin{array}{ll}
a_1^{(k)} &=1\\
a_l^{(k)} &=a_{l-1}^{(k-1)}+la_l^{(k-1)}\qquad l=2,\dots, k-1\\
a_k^{(k)} &=1
\end{array}
\end{equation}
The last recursion follows from the fact that 
$$M'_{k,n}(s) = -k M_{k,n}(s) - M_{k+1,n}(s)$$
From proposition \ref{crucial} and equation (\ref{allmoments}), we have the following theorem.
\begin{thm}
If the $(p_1,\ldots,p_n)$ are determined by normalizing the increments of a $\gamma$-stable subordinator with $t_i-t_{i-1}=1/n$ in \eqref{eq:lapl_w} for each $i \in \{1, \ldots, n\}$, then
$$
\lim_{n \rightarrow \infty} E(S_n^k) = \left\{ \begin{array}{ll}
\sum_{l=1}^k \frac{(l!)^2}{(\frac{1}{\gamma}-l-1)_{l}} a_l^{(k)}&\mbox{ if }\gamma<\frac{1}{k+1}\\ \infty &{\mbox{otherwise}}\\
\end{array}\right.
$$
\end{thm}
The previous theorem allows to calculate all the moments of the limiting search cost distribution in the stable case. For example the second moment is 
\begin{eqnarray*}
\lim_{n \rightarrow \infty} E(S_n^2) 
& = & \left\{ \begin{array}{ll}
\frac{\gamma(1+\gamma)}{(1-3\gamma)(1-2\gamma)}&\mbox{ if }\gamma<\frac{1}{3}\\ \infty &{\mbox{otherwise}}\\
\end{array}\right.\\
\end{eqnarray*}
and the third moment 
\begin{eqnarray*}
\lim_{n \rightarrow \infty} E(S_n^3) 
& = & \left\{ \begin{array}{ll}
\frac{\gamma(1+5\gamma)}{(1-4\gamma)(1-3\gamma)(1-2\gamma)}&\mbox{ if }\gamma<\frac{1}{4}\\ \infty &{\mbox{otherwise}}\\
\end{array}\right.\\
\end{eqnarray*}

%




\end{document}